\documentclass[11pt]{amsart}
\usepackage{amssymb, amstext, amscd, amsmath}
\usepackage[all]{xy}

\theoremstyle{plain}
\newtheorem{theorem}{Theorem}[section]

\newtheorem{proposition}[theorem]{Proposition}
\newtheorem{lemma}[theorem]{Lemma}

\theoremstyle{definition}

\newtheorem{example}[theorem]{Example}

\newtheorem{remark}[theorem]{Remark}

\newtheorem*{acknow}{Acknowledgements}

\theoremstyle{definition}

\makeatletter
\def\@cite#1#2{{\m@th\upshape\bfseries%
[{#1\if@tempswa{\m@th\upshape\mdseries, #2}\fi}]}} \makeatother

\newcommand{\bbC}{{\mathbb{C}}}
\newcommand{\bbD}{{\mathbb{D}}}

\newcommand{\bbT}{{\mathbb{T}}}
\newcommand{\bbZ}{{\mathbb{Z}}}

\newcommand{\A}{{\mathcal{A}}}
\newcommand{\B}{{\mathcal{B}}}

\newcommand{\F}{{\mathcal{F}}}
\newcommand{\G}{{\mathcal{G}}}

\newcommand{\K}{{\mathcal{K}}}
\renewcommand{\L}{{\mathcal{L}}}

\renewcommand{\O}{{\mathcal{O}}}

\newcommand{\T}{{\mathcal{T}}}

\newcommand{\be}{{\mathbf{e}}}

\newcommand{\rC}{{\mathrm{C}}}

\renewcommand{\phi}{\varphi}
\newcommand{\upchi}{{\raise.35ex\hbox{\ensuremath{\chi}}}}

\def\ga{\alpha}

\def\be{\beta}

\newcommand{\foral}{\text{ for all }}

\newcommand{\alg}{{\operatorname{alg}}}
\newcommand{\id}{{\operatorname{id}}}
\newcommand\Span{\mathop{\rm span}}
\newcommand\cov{\mathop{\rm cov}}

\newcommand{\ca}{\mathrm{C}^*}
\newcommand{\cenv}{\mathrm{C}^*_{\textup{env}}}

\newcommand{\sca}[1]{\left\langle#1\right\rangle} 
\newcommand{\lsca}[1]{\left[#1\right]}            
\newcommand{\vrt}{\G^{(0)}}

\newcommand{\edg}{\G^{(1)}}
\newcommand{\xtau}{X_{\tau}}

\addtocontents{toc}{\protect\setcounter{tocdepth}{1}}

\begin{document}

\title{The Dirichlet Property for Tensor Algebras}

\author[E.T.A. Kakariadis]{Evgenios T.A. Kakariadis}
\address{Pure Mathematics Department, University of Waterloo,
   Ontario N2L-3G1, Canada}
\email{ekakaria@uwaterloo.ca}

\thanks{The author was partially supported by the Fields Institute for Research in the Mathematical Sciences}

\maketitle

\begin{abstract}
We prove that the tensor algebra of a $\ca$-correspondence $X$ is Dirichlet if and only if $X$ is a Hilbert bimodule. As a consequence, we point out and fix an error appearing in the proof of a famous result of Duncan. Secondly we answer a question raised by Davidson and Katsoulis concerning tensor algebras and semi-Dirichlet algebras, by giving an example of a Dirichlet algebra that cannot be described as the tensor algebra of any $\ca$-correspondence. Furthermore we show that the adding tail technique, as extended by the author and Katsoulis, applies in a unique way to preserve the class of Hilbert bimodules.

The exploitation of these ideas implies that the tensor algebra of row-finite graphs, the tensor algebra of multivariable automorphic $\ca$- dynamics and Peters' semicrossed product of an injective $\ca$-dynamical system  have the unique extension property. The two latter provide examples of non-separable operator algebras that admit a Choquet boundary in the sense of Arveson.
\end{abstract}

\section*{Introduction}

In this paper we settle some questions raised in the context of tensor algebras of $\ca$-correspondences and semi-Dirichlet algebras. The key result is that the tensor algebra of a $\ca$-correspondence $X$ is Dirichlet if and only if $X$ is a Hilbert bimodule (Theorem \ref{T:main 3}). Our purpose is to underline its consequences.

First, we point out an error in \cite[Proposition 3]{Dun08}, which can be corrected: Peters' semicrossed product of an injective $\ca$-dynamical system may not be always Dirichlet but it has the unique extension property (Theorem \ref{T:main 5}). Moving even further we show that the tensor algebras of row-finite graphs or multivariable automorphic $\ca$-dynamics have also the unique extension property (Theorem \ref{T:graphs} and Theorem \ref{T:auto}). Recall that if an operator algebra has the unique extension property then it admits a Choquet boundary in the sense of Arveson \cite{Ar08}, even when it is non-separable. We remark that semicrossed products and multivariable dynamical systems have been under considerable investigation for the last four decades (e.g., \cite{Arv1,ArvJ,DavKak11,DK08,DKsimple,DavKat,Dun08,HadH,McAsMuhSai79,MS00,Pet2,Pow} to mention but a few). Recently Cornelissen and Marcolli provide a link that connects the theory of Davidson and Katsoulis \cite{DavKat} for multivariable (automorphic) dynamical systems with number theory \cite{C,CM} and reconstruction of graphs \cite{CM2}.

As a second consequence we answer a question raised in \cite{DavKat11-2}. Davidson and Katsoulis \cite{DavKat11-2} examine the dilation theory of operator algebras along with versions of a commutant lifting theorem such as finite dimensional nest algebras, tensor algebras of $\ca$-correspondences, bilateral tree algebras etc. Moreover they examine the dilation theory of the class of semi-Dirichlet algebras. Tensor algebras is a subclass of the semi-Dirichlet algebras and the question raised in \cite{DavKat11-2} was whether these two classes coincide. We answer this question to the negative here by giving an example of an operator algebra that is semi-Dirichlet (even more it is Dirichlet) but cannot be described as the tensor algebra of any $\ca$-correspondence.

Finally, we give an application to the ``adding tail'' technique. In \cite{KakKats11} the author and Katsoulis extend the construction of Muhly and Tomforde \cite{MuTom04}: given a non-injective $\ca$-correspondence one can produce injective $\ca$-correspondences $Y$ such that the Cuntz-Pimsner algebras $\O_X$ and $\O_Y$ are Morita equivalent, by adding (a variety of) appropriate tails. In this way one can add an appropriate tail such that $X$ and $Y$ live in the same sub-class of $\ca$-correspondences, e.g., in the class of semicrossed products. In this paper we show that this technique respects the class of Hilbert bimodules. Even more, when restricted to this sub-class of $\ca$-correspondences, the tail has a unique form (Theorem \ref{T:main 4}).

The structure of the paper is as follows. In Section \ref{S:preliminaries} we briefly discuss the elements of the theory that we use. Since these subjects are by now well known we omit full details. In Section \ref{S:dir pro} we prove the key result. In Section \ref{S:app} we give the applications concerning the adding tail technique, a discussion on \cite[Proposition 3]{Dun08}, the unique extension property of various (non-separable) operator algebras and the counterexample related to \cite{DavKat11-2}.

\section{Preliminaries}\label{S:preliminaries}

We will require terminology concerning (non-selfadjoint) operator algebras and $\ca$-correspondences; for more details see \cite{Pau02} and \cite{Lan95}, respectively. Every representation is assumed to act on a Hilbert space.

Let $\A$ be an operator algebra and $\rho\colon \A \rightarrow \B(H)$ a completely contractive representation. A \emph{dilation} $\nu\colon \A \rightarrow \B(K)$ of $\rho$ is a completely contractive representation such that $P_H \nu(\cdot)|_H = \rho$. A representation is called \emph{maximal} if it has no non-trivial dilations. By \cite{DrMc05} every completely isometric representation has a maximal dilation. The same holds for a completely contractive representation $\rho$, by considering the maximal dilation of the direct sum of $\rho$ with a maximal completely isometric representation.

For a completely isometric homomorphism $j \colon \A \rightarrow C= \ca(j(\A))$, the pair $(C,j)$ is called \emph{a $\ca$-cover for $\A$}. The \emph{$\ca$-envelope $\cenv(\A) \equiv (\cenv(\A),\iota)$} of $\A$ is the (universal) $\ca$-cover with the following property: for any $\ca$-cover $(C,j)$ there is a $*$-epimorphism $\Phi\colon C \rightarrow \cenv(\A)$ such that $\Phi(j(a))=\iota(a)$ for all $a\in \A$. For the existence of the $\ca$-envelope see \cite{Ham79,DrMc05}.

An operator algebra $\A$ is called \emph{Dirichlet} if $\iota(\A) + \iota(\A)^*$ is dense in the $\ca$-envelope $\cenv(\A)$. If $\iota(\A)^*\iota(\A) \subseteq \iota(\A) + \iota(\A)^*$, then $\A$ is called \emph{semi-Dirichlet}. An operator algebra $\A$ has \emph{the unique extension property} if the restriction of every faithful representation of $\cenv(\A)$ to $\A$ is maximal. In this case $\A$ has automatically a \emph{Choquet boundary} in the sense of Arveson \cite{Ar08} (even when $\A$ is non-separable); for a quick proof consider the free atomic representation of $\cenv(\A)$. Recall that Arveson \cite{Ar08} proves the existence of the Choquet boundary only for separable operator algebras.

A \emph{$\ca$-correspondence $X_A$ over a $\ca$-algebra $A$} is a right Hilbert $A$-module together with a $*$-homomorphism $\phi_X\colon A \rightarrow \L(X)$.
A (\emph{Toeplitz}) \emph{representation} of $X$ into a $\ca$-algebra $B$, is a pair $(\pi,t)$, where $\pi\colon A \rightarrow B$ is a $*$-homomorphism and $t\colon X \rightarrow B$ is  a linear map, such that $\pi(a)t(\xi)=t(\phi_X(a)(\xi))$ and $t(\xi)^*t(\eta)=\pi(\sca{\xi,\eta}_X)$, for all $a\in A$ and $\xi,\eta\in X$. The $\ca$-identity implies that $t(\xi)\pi(a)=t(\xi a)$. A representation $(\pi, t)$ is called \textit{injective} if $\pi$ is injective; in that case $t$ is an isometry.
The $\ca$-algebra generated by a representation $(\pi,t)$ equals the closed linear span of $t^n(\bar{\xi})t^m(\bar{\eta})^*$, where $\bar{\xi}\equiv \xi_1 \otimes \cdots \otimes \xi_n \in X^{\otimes n}$ and $t^n(\bar{\xi})\equiv t(\xi_1)\dots t(\xi_n)$. For any representation $(\pi,t)$ there exists a $*$-homomorphism $\psi_t\colon \K(X)\rightarrow B$, such that $\psi_t(\Theta^X_{\xi,\eta})= t(\xi)t(\eta)^*$ \cite{KajPinWat98}.

Let $J$ be an ideal in $\phi_X^{-1}(\K(X))$; a representation $(\pi,t)$ is called \emph{$J$-coisometric} if $\psi_t(\phi_X(a))=\pi(a)$, for all $a\in J$.
Following Katsura \cite{Kats04}, the representations $(\pi,t)$ that are $J_{X}$-coisometric, where $J_X=\ker\phi_X^\bot \cap \phi_X^{-1}(\K(X))$, are called \emph{covariant representations}. We denote by $\cov(X)$ the family of injective pairs $(\pi,t)$ that admit a gauge action $\{\be_z\}_{z\in \bbT}$.

The \emph{Toeplitz-Cuntz-Pimsner algebra} $\T_X$ is the universal $\ca$-algebra for ``all'' representations of $X$, and the \emph{Cuntz-Pimsner algebra} $\O_X$ is the universal $\ca$-algebra for ``all'' covariant representations of $X$. There is a well known representation $(\rho,s)$ acting on the Fock space $\F_{X} : = \oplus_n X^{\otimes n}$. Fowler and Raeburn \cite{FR} (resp. Katsura \cite{Kats04}) prove that the $\ca$-algebra $\ca( \rho, s)$ (resp. $\ca(\rho, s)/ \K(\F_{X}J_{X})$) is $*$-isomorphic to $\T_{X}$ (resp. $\O_{X}$).

The \emph{tensor algebra} $\T_{X}^+$ is the norm-closed algebra generated by the universal copy of $A$ and $X$ in $\T_X$. Examples of tensor algebras are Peters' semicrossed product \cite{Pet2}, Popescu's non-commutative disc algebras \cite{Pop3}, the tensor algebras of graphs \cite{MS} and the tensor algebras for multivariable dynamics \cite{DavKat,MS}. Katsoulis and Kribs \cite[Theorem 3.7]{KatsKribs06} prove that the $\ca$-envelope of the tensor algebra $\T^+_X$ is $\O_X$.

A \emph{Hilbert bimodule} $X_A$ is a $\ca$-correspondence together with a left inner product $\lsca{\cdot,\cdot}_X\colon X \times X \rightarrow A$, such that $\lsca{\xi,\eta}_X \cdot \zeta=\xi \cdot \sca{\eta,\zeta}_X$ for all $\xi, \eta,\zeta \in X$. This compatibility relation implies that $J_X= \overline{\Span} \{ \lsca{\xi,\eta}_X \mid \xi,\eta \in X\}$ and that $\phi_X(\lsca{\xi,\eta}_X) = \Theta^X_{\xi,\eta}$. Thus $\phi_X$ is injective if and only if the Hilbert bimodule $X_A$ is \emph{essential}, i.e., when the ideal $J_X$ is essential in $A$. When $J_X=A$ then $X_A$ is called an \emph{imprimitivity bimodule}.

\section{The Dirichlet Property for Tensor Algebras}\label{S:dir pro}

It is immediate that Dirichlet algebras have the unique extension property and are semi-Dirichlet. On the other hand tensor algebras of $\ca$-correspondences are also semi-Dirichlet, since $\cenv(\T_X^+) \simeq \O_X$. Here we show that $\T_X^+$ is in particular Dirichlet if and only if $X_A$ is a Hilbert bimodule. Let us start with a general lemma.

\begin{lemma}\label{L:equiv via comp}
Let $X_A$ be a $\ca$-correspondence. Then the following are equivalent
\begin{enumerate}
\item[\textup{(1)}] $X_A$ is a bimodule;
\item[\textup{(2)}] $K(X) \subseteq \phi_X(A)$ and $\phi_X^{-1}(\K(X))=\ker\phi_X \oplus J_X$, as an orthogonal sum of ideals;
\item[\textup{(3)}] the restriction of $\phi_X$ to $J_X$ is a $*$-isomorphism onto $\K(X)$.
\end{enumerate}
In particular, if $\phi_X$ is injective then $X$ is a Hilbert bimodule if and only if $\K(X) \subseteq \phi_X(A)$.
\end{lemma}
\begin{proof}
The equivalence $(1) \Leftrightarrow (3)$ is due to the compatibility relation that characterizes Hilbert bimodules (for example see \cite{Kats03}).

Assume that item (2) holds. By definition of  $J_X$ the restriction of $\phi_X$ to $J_X$ is injective. Moreover it is also onto $\K(X)$ since
\begin{align*}
\K(X)=\phi_X \circ \phi_X^{-1}(\K(X))= \phi_X(\ker\phi_X \oplus J_X)= \phi_X(J_X),
\end{align*}
which gives the implication $(2) \Rightarrow (3)$.

For the converse, first observe that $\K(X)= \phi_X(J_X) \subseteq \phi_X(A)$. Let $a \in \phi_X^{-1}(\K(X))$, i.e., $\phi_X(a)=k \in \K(X)$. Since, $\phi_X(J_X)=\K(X)$, there is an $x\in J_X$ such that $\phi_X(x)=k$. Thus $\phi_X(a-x)=0$, hence $a-x \in \ker\phi_X$. So $a\in \ker\phi_X + J_X$ which implies that $\phi^{-1}_X(\K(X)) \subseteq \ker\phi_X + J_X$. By definition of $J_X$ the sum is orthogonal.
\end{proof}

\begin{theorem}\label{T:main 3}
For a $\ca$-correspondence $X_A$ the following are equivalent:
\begin{enumerate}
\item[\textup{(1)}] $X_A$ is a Hilbert bimodule;
\item[\textup{(2)}] $\psi_t(\K(X)) \subseteq \pi(A)$, for $(\pi,t)\in \cov(X)$;
\item[\textup{(3)}] the tensor algebra $\T_X^+$ is Dirichlet.
\end{enumerate}
\end{theorem}
\begin{proof}
The equivalence $(1) \Leftrightarrow (2)$ is known \cite[Proposition 5.18]{Kats04}. Assume that item (1) holds and fix $(\pi,t) \in \cov(X)$. Then item (2) holds thus $t^n(\bar{\xi})t^n(\bar{\eta})^* \in \pi(A)$ for all $\bar{\xi},\bar{\eta} \in X^{\otimes n}$ and $n\in \bbZ_+$, inductively. Hence,
\begin{align*}
\O_X &= \overline{\Span}\{t^n(\bar{\xi})t^m(\bar{\eta})^*\mid \bar{\xi}
\in X^{\otimes n}, \bar{\eta} \in X^{\otimes m}, n,m \in
\bbZ_+\}\\
& = \overline{\Span}\{ t^n(\bar{\xi}) + t^m(\bar{\eta})^* \mid \bar{\xi}
\in X^{\otimes n}, \bar{\eta} \in X^{\otimes m}, n,m \in
\bbZ_+\}\\
&= \overline{ \T_X^+ + (\T_X^+)^* },
\end{align*}
where we have used that
\begin{align*}
t^n(\bar{\xi}) t^m(\bar{\eta})^* =
\begin{cases}
t^{n-m}(\bar{\xi'})\pi(a) &, \text{ when } n>m,\\
\pi(a) &, \text{ when } n=m,\\
(t^{m-n}(\bar{\eta'})\pi(a))^* &, \text{ when } n<m.
\end{cases}
\end{align*}
for appropriate $a\in A, \bar{\xi'} \in X^{\otimes n-m}, \bar{\eta'} \in X^{\otimes m-n}$, since $\psi_{t^n}(\K(X^{\otimes n})) \subseteq \pi(A)$. This gives the implication $(1) \Rightarrow (3)$.

For the converse, assume that $\T_X^+$ has the Dirichlet property. Then
\begin{align*}
\overline{\Span} & \{ t^n(\bar{\xi}) + t^m(\bar{\eta})^*\mid \bar{\xi}
\in X^{\otimes n}, \bar{\eta} \in X^{\otimes m}, n,m \in
\bbZ_+\}= \\
& \hspace{15mm} = \overline{ \T_X^+ + (\T_X^+)^* }\\
& \hspace{15mm} = \O_X\\
& \hspace{15mm} = \overline{\Span}\{t^n(\bar{\xi})t^m(\bar{\eta})^* \mid
\bar{\xi} \in X^{\otimes n}, \bar{\eta} \in X^{\otimes m}, n,m \in \bbZ_+\}.
\end{align*}
Let $\{\beta_z\}_{z\in \bbT}$ be the gauge action for $(\pi,t)$ and $E\colon \O_X \rightarrow \O_X^\beta$ be the conditional expectation induced by $\{\beta_z\}_{z\in \bbT}$ such that $E(F)= \int_\bbT \beta_z(F) dz$, for $F\in \O_X$. By applying $E$ to the above equality we deduce
\begin{align*}
\pi(A)&= E\bigg(\overline{\T_X^+ + (\T_X^+)^*} \bigg)\\
& = E( \O_X)\\
& = \overline{\Span}\{t^n(\bar{\xi})t^n(\bar{\eta})^*\mid \bar{\xi}, \bar{\eta} \in X^{\otimes n}, n \in \bbZ_+\}\\
& = \overline{ \pi(A) + \{ \psi_{t^n}(\K(X^{ \otimes n}))\mid n
\in \bbZ_+\} }.
\end{align*}
In particular $\psi_t(\K(X)) \subseteq \pi(A)$, hence $X$ is a Hilbert bimodule.
\end{proof}

\section{Applications}\label{S:app}

\subsection{Adding Tails to Hilbert Bimodules}\label{S:add tail}

In \cite{KakKats11} the author and Katsoulis extended the method of ``adding tails''  introduced by Muhly and Tomforde \cite{MuTom04} in a way that it preserves sub-classes of $\ca$-correspondences. We give a brief description using the more elegant notion of graph correspondences of Deaconu, Kumjian, Pask and Sims \cite{DKPS10}, rather than the language used in \cite{KakKats11}. We use \cite{Raeb} as a general reference for graphs.

Let $\G=(\vrt,\edg,r,s)$ be a row-finite graph, i.e., $|r^{-1}(p)| < \infty$ for all vertices $p\in \G^{(0)}$. Let $(A_p)_{p \in \vrt}$ be a family of $\ca$-algebras and for each $e \in \edg$, let $X_e$ be a $A_{r(e)}$-$A_{s(e)}$-correspondence. Let $A_\G= c_0 (\, (A_p)_{p \in \vrt})$ denote the $c_0$-sum of the family $(A_p)_{p \in \vrt}$. Also, let $Y_0=c_{00}((X_e)_{e \in \edg})$ which is equipped with the $A_\G$-valued inner product
\begin{align*}
&\sca{ u, v  }(p)= \sum_{s(e)=p} \, \sca{u_e,v_e}_{A_p}, &(p \in
\vrt),
\end{align*}
If $X_\G$ is the completion of $Y_0$ with respect to the inner product, then $X_\G$ is $\ca$-correspondence over $A_\G$ when equipped with the actions
\begin{align*}
&(ux)_e = u_ex_{s(e)},  & (e \in \edg),\\
&(\phi_{\G}(x)u)_e= \phi_e(x_{r(e)})(u_e), & (e\in \edg),
\end{align*}
for $x\in A_\G$ and $u \in X_\G$. The $\ca$-correspondence $X_\G$ over $A_\G$ is called \emph{the graph correspondence associated to}
$\big\{\G, \{A_p\}_{p \in \vrt}, \{X_e\}_{e\in \edg} \big\}$.

Every $\ca$-correspondence $X_A$ can be viewed as the following graph correspondence
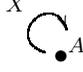
\begin{align*}
\xymatrix{ \bullet^A \ar@(l,u)[]^X }
\end{align*}

Then Muhly-Tomforde adding tail technique \cite{MuTom04} produces the
following graph correspondence
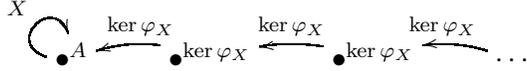
\begin{align*}
\xymatrix{ \bullet^A \ar@(l,u)[]^X & \bullet^{\ker\phi_X}
\ar@/_/[l]_{\ker\phi_X} & \bullet^{\ker\phi_X}
\ar@/_/[l]_{\ker\phi_X} & \cdots \ar@/_/[l]_{\ker\phi_X}}
\end{align*}

More generally, to a non-injective $\ca$-correspondence $X_A$ we can ``add a tail'' on the distinguished vertex $p_0$ of the cycle graph to obtain
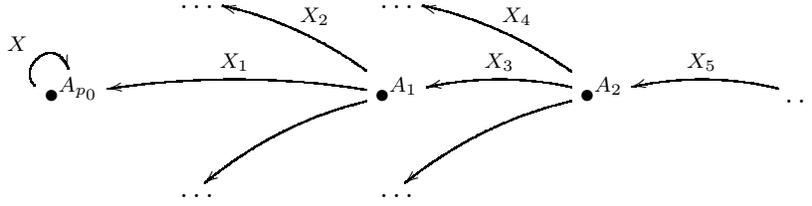
\begin{align*}
\xymatrix{& \dots & & \dots & & & & \\
\bullet^{A_{p_0}} \ar@(l,u)[]^{X} & & & \bullet^{A_1}
\ar@/_/[lll]_{X_1} \ar@/_/[ull]_{X_2} \ar@/_/[dll] & & \bullet^{A_2}
\ar@/_/[ll]_{X_3} \ar@/_/[ull]_{X_4} \ar@/_/[dll] & & \dots
\ar@/_/[ll]_{X_5}\\
& \dots & & \dots & & & &}
\end{align*}
where $A_{p_0} \equiv A$. We also have the following requirements:

\noindent $\bullet$ For $e \neq e_1, e_0$ each $X_e$ is an $A_{r(e)}$-$A_{s(e)}$-equivalence bimodule,

\noindent $\bullet$ For $e = e_1$, $X_1$ is a full $A_1$-module with $\K(X_{1}) \subseteq \phi_{X_1} (A)$ and
\begin{align*}
J_{X}\subseteq \ker \phi_{X_1} \subseteq \left( \ker \phi_X
\right)^{\perp},\quad \phi_{X_1}^{-1}(\K(X_{1})) \subseteq
\phi_X^{-1}(\K(X))
\end{align*}

\noindent $\bullet$ When we exclude the cycle on $p_0$, the graph is $p_0$-accessible, has no sources, and there is one infinite path $w$ such that $r(w)=p_0$ \cite[Theorem 7.3]{KakKats11}.

A graph correspondence of the above form will be denoted by $\xtau$.
By \cite[Theorem 3.10]{KakKats11} $\xtau$ is an injective $\ca$-correspondence and the Cuntz-Pimsner algebra $\O_{\xtau}$ is a full corner of $\O_X$. We mention that the properties listed above are necessary for this purpose \cite[Proposition 3.13]{KakKats11}.

Therefore given a non-injective $\ca$-correspondence $X_A$ one can produce a family of injective $\ca$-correspondences with Morita equivalent Cuntz-Pimsner algebras, by choosing different tails of the above form. However, under the constraint that $\xtau$ is an (essential) Hilbert bimodule, there is a unique form of such tails.

\begin{theorem}\label{T:main 4}
Let $X$ be a non-injective $\ca$-correspondence. Then a graph correspondence $X_\tau$ as defined above is an $($essential$)$ Hilbert bimodule if and only if $X$ is a Hilbert bimodule and $|s^{-1}(p)|=|r^{-1}(p)|=1$ for every $p \neq p_0$.
\end{theorem}
\begin{proof}
If $\xtau$ is a Hilbert bimodule then $\T^+_{\xtau}$ is a Dirichlet algebra, by Theorem \ref{T:main 3}. Hence, if $Q$ is the projection provided by \cite[Theorem 3.10]{KakKats11} such that $\O_X = Q \O_{\xtau} Q$, then
\begin{align*}
\O_X & = Q \O_{\xtau} Q
  = Q \bigg(\overline{ \T_{\xtau}^+ + (\T_{\xtau}^+)^*} \bigg)Q \\
& = \overline{ Q(\T_{\xtau}^+) Q + Q(\T_{\xtau}^+)^* Q }
  = \overline{\T_X^+ + (\T_X^+)^*},
\end{align*}
where the last equation is deduced by \cite[Lemma 3.7]{KakKats11}. Thus $\T_X^+$ is also a Dirichlet algebra. Hence, $X$ is a Hilbert bimodule. For the second part, assume that there is a $p \neq p_0$ such that $|s^{-1}(p)| \geq 2$. Then at the vertex $p$ we would have (at least) the following picture
\begin{align*}
\xymatrix{ A_q  & \\
A_w  & A_p \ar@/_/_{X_r}[ul] \ar_{X_f}[l]}
\end{align*}
Let $u_r \in X_r$ and $u_f \in X_f$. Then
\begin{align*}
\Theta^{\xtau}_{u_r \chi_r, u_f \chi_f} (v_f \chi_f)= (u_r
\sca{u_f,v_f}_{A_p})\chi_r, \foral v_f \in X_f.
\end{align*}
Since $X_r \sca{X_f,X_f}_{A_p}$ cannot be zero (as $\sca{X_f,X_f} = A_p$), we can choose $u_r$ and $u_f$ such that the compact operator $\Theta^{\xtau}_{u_r \chi_r, u_f \chi_f}$ is not trivial. Since $\xtau$ is a bimodule there is an $(a_e) \in A_\tau$ such that $\phi_\tau( (a_e))= \Theta^{\xtau}_{u_r \chi_r, u_f \chi_f}$. But then
\begin{align*}
X_f \ni \phi_{X_f} (a_w) v_f \chi_f
& = \phi_\tau( (a_e) ) v_f\chi_f\\
& = \Theta^{\xtau}_{u_r \chi_r, u_f \chi_f} (v_f \chi_f)\\
& = (u_r\sca{u_f,v_f}_{A_p})\chi_r \in X_r,
\end{align*}
for all $v_f \in X_f$, which is absurd. On the other hand if there was a vertex $p$ such that $|r^{-1}(p)| \geq 2$ then we would have (at least) the following picture
\begin{align*}
\xymatrix{ A_q \ar@/^/^{X_r}[dr]  & \\
A_w \ar_{X_f}[r] & A_p}
\end{align*}
Pick $u_r, v_r \in X_r$ such that $u_r\sca{v_r,v_r} \neq 0$. Since $\xtau$ is a bimodule, then there is an element $(a_e) \in A_\tau$ such that $\phi_\tau( (a_e) )= \Theta^{\xtau}_{ u_r \chi_r, v_r \chi_r}$, for $u_r, v_r \in X_r$. Hence,
\begin{align*}
\phi_\tau( (a_e) ) (v_f \chi_f)= \Theta^{\xtau}_{ u_r \chi_r, v_r
\chi_r}(v_f \chi_f)= u_r\chi_r \sca{v_r \chi_r, v_f \chi_f}_{\xtau}= 0,
\end{align*}
for all $v_f \in X_f$. But $\phi_\tau((a_e)) (v_f)= \phi_{X_f}(a_p\chi_p) (v_f)$. Thus $a_p \in \ker\phi_{X_f}$, which is the trivial ideal since $X_f$ is an equivalence bimodule, and so $a_p=0$. In the same time
\begin{align*}
0
  = \phi_{X_r}(a_p) v_r \chi_r
&  = \phi_\tau( (a_e) ) (v_r \chi_r)\\
& = \Theta^{\xtau}_{ u_r \chi_r, v_r \chi_r}(v_r \chi_r)\\
& = u_r\chi_r \sca{v_r \chi_r, v_r \chi_r}_{\xtau}
  = u_r\sca{v_r, v_r}_{X_r} \chi_r
\end{align*}
which is a contradiction.

For the converse, in view of Lemma \ref{L:equiv via comp} and since $\xtau$ is injective, it suffices to prove that $\K(\xtau) \subseteq \phi_\tau(A_\tau)$. Since $|s^{-1}(p)|=|r^{-1}(p)|=1$ for every $p \neq p_0$ it is straightforward that $\K(X_e, X_f)= (0)$ for $e\neq f$, when $X_e$ and $X_f$ are viewed as Hilbert submodules of $\xtau$. Thus $\K(\xtau)$ is the closure of the linear span of operators in $\K(X_e)$ for $e \in \edg$. In particular $\xtau$ can be written as the following graph correspondence
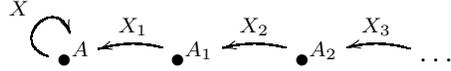
\begin{align*}
\xymatrix{ \bullet^A \ar@(l,u)[]^X & \bullet^{A_1} \ar@/_/[l]_{X_1}
& \bullet^{A_2} \ar@/_/[l]_{X_2} & \cdots \ar@/_/[l]_{X_3}}
\end{align*}
hence $\phi_\tau( a,a_1, a_2, \dots) = (\phi_X(a), \phi_{X_1}(a), \phi_{X_2}(a_1), \dots)$. For the imprimitivity bimodules $X_n \neq X, X_1$ and $k \in \K(X_n)$ there is an $a_{n-1} \in A_{n-1}$ such that $\phi_n (a_{n-1})= k$. Thus,
\begin{align*}
(k \chi_n)(\xi, (u_n))= k u_n \chi_n= \phi_n(a_{n-1}) u_n \chi_n=
\phi_\tau( a_{n-1} \chi_{n-1})(\xi, (u_n)),
\end{align*}
for all $\xi \in X, u_n \in X_n$, therefore $\K(X_n) \subseteq \phi_\tau(A_\tau)$ for all $n \neq 1$. Also, since $X$ is a Hilbert bimodule, for $k \in \K(X)$ there is an $a\in J_X$ such that $\phi_X(a)=k$. By the linking condition $J_X \subseteq \ker \phi_{X_1}$,
for $X$ and $X_1$ in the definition of $X_\tau$ we have that $\phi_{X_1}(a)=0$. Hence,
\begin{align*}
k(\xi, (u_n))= k \xi \chi_0= \phi_X(a) \xi \chi_0=
\phi_\tau( a, 0)(\xi, (u_n)),
\end{align*}
for all $\xi \in X, u_n \in X_n$, thus $\phi_\tau(a,0)=k$ and so $\K(X) \subseteq \phi_\tau(A_\tau)$. Finally, let a compact operator $k\in \K(X_1)$. Since
\[
\K(X_1) \subseteq \phi_{X_1}(A) \text{ and } \phi^{-1}_{X_1}(\K(X_1)) \subseteq \phi^{-1}_X( \K(X)),
\]
by definition of $\xtau$, there is an $a\in \phi^{-1}_X(\K(X))$ such that $\phi_{X_1}(a)= k$. Since $X$ is a Hilbert bimodule we can write $a= b +c$ for some $b\in \ker\phi_X$ and $c\in J_X$, by Lemma \ref{L:equiv via comp}. Recall that $J_X \subseteq \ker\phi_{X_1}$, thus $k=\phi_{X_1}(a)=\phi_{X_1}(b) + \phi_{X_1}(c)= \phi_{X_1}(b)$. Hence,
\begin{align*}
\phi_\tau(b,(0))= (\phi_X(b), \phi_{X_1}(b), (0))= (0, \phi_{X_1}(b), (0) ) = (0, \phi_{X_1}(a), (0) )= k \chi_1,
\end{align*}
therefore $\K(X_1) \subseteq \phi_\tau(A_\tau)$. Thus $\K(\xtau) \subseteq \phi_\tau(A_\tau)$ and the proof is complete.
\end{proof}

\subsection{Graphs and $\ca$-dynamical systems}

Let us apply Theorem \ref{T:main 4} to three fundamental examples of $\ca$-correspondences associated to row-finite graphs, $\ca$-dynamical systems and multivariable $\ca$-dynamics.

In the context of graph theory, the Cuntz-Pimsner algebra $\O_\G$ associated to a graph is exactly the Cuntz-Krieger algebra $\ca(\G)$ \cite{BatHonRaeSzy02,RaeSzy04}, i.e., the universal $\ca$-algebra generated by a set of mutual orthogonal projections $p_v, v\in \G^{(0)}$ and a set of partial isometries $s_e, e\in \G^{(1)}$ such that
$s_e^*s_e  = p_{s(e)}$, for all $e\in \G^{(1)}$, and $p_v  = \sum_{e \in r^{-1}(v) } s_es_e^*$, when $0< |r^{-1}(v)| < \infty$. For a full discussion the reader is addressed to \cite[Section 8]{Raeb}.

\begin{example}\label{E:graph dirichlet}
\textup{Consider a directed graph $\G=(\vrt,\edg,r,s)$ and form the $\ca$-correspondence $X_\G$ of the graph $\G$ (that is the graph correspondence associated to $\big\{\G, \{\bbC_p\}_{p \in \vrt}, \{\bbC_e\}_{e\in \edg} \big\}$). An element of $A_\G$ is in the kernel of the left action if and only if it is a source, that is $r^{-1}(p)=\emptyset$. Then, by Lemma \ref{L:equiv via comp} and Theorem \ref{T:main 4}, the tensor algebra $\T_{X_\G}^+$ is Dirichlet if and only the vertices of the graph emit and receive one or none edge. The only possible cases for $\G$ is then to be a finite path, an one-sided infinite path, a two-sided infinite path or a circular graph. (For example, the non-commutative disc algebras $\A_n$, are not Dirichlet.)}
\end{example}

Nevertheless the tensor algebra of a row-finite graph has the (weaker) unique extension property.

\begin{theorem}\label{T:graphs}
The tensor algebra associated to a row-finite graph has the unique extension property.
\end{theorem}
\begin{proof}
Let a Cuntz-Krieger family $\{p_v, s_e \mid v\in \G^{(0)}, e\in \G^{(1)}\}$ that integrates to a faithful representation of $\O_\G$. Then $\T_\G^+ = \alg\{p_v, s_e \mid v\in \G^{(0)}, e\in \G^{(1)}\}$. Let a faithful representation $\rho\colon \O_\G \rightarrow \B(H)$ and a maximal dilation $\nu \colon \T_G^+ \rightarrow \B(K)$ of $\rho|_{\T_\G^+}$. The dilation $\nu$ extends uniquely to a faithful representation of $\O_\G$, which will be denoted by the same letter. Therefore the families
\[
\{ \rho(p_v), \rho(s_e) : v\in \vrt, e\in \edg\} \text{ and } \{ \nu(p_v), \nu(s_e) : v\in \vrt, e\in \edg\}
\]
are Cuntz-Krieger families. By \cite{DrMc05} it suffices to show that $H$ is $\nu(\T_G^+)$-invariant, i.e., $\nu$ is a trivial dilation.

First we remark that $\nu|_{\ca(p_v \mid v\in \G^{(0)})}$ is trivial as a dilation of the $\ca$-algebra $\ca(p_v \mid v\in \G^{(0)})$. Hence
\[
\nu(p_v) = \begin{bmatrix} \rho(p_v) & 0 \\ 0 & \ast \end{bmatrix},
\]
for all $v\in \G^{(0)}$. Let $e\in \G^{(1)}$ and assume that
\[
\nu(s_e) = \begin{bmatrix} \rho(s_e) & a \\ b & \ast \end{bmatrix}.
\]
Since $\nu$ is a representation of $\O_\G$ we obtain that $\nu(s_e)^*\nu(s_e) = \nu(p_{s(e)})$ and by equating the $(1,1)$-entries we get that $\rho(s_e)^*\rho(s_e) + b^*b = \rho(p_{s(e)})$.
But $\rho(s_e)^*\rho(s_e) = \rho(p_{s(e)})$ hence $b=0$. Also $\G$ is row-finite, hence there is a vertex $v$ such that $e \in r^{-1}(v)$ with $|r^{-1}(v)|< \infty$. Therefore there are edges $e_1, \dots, e_n$ such that $p_v= s_es_e^* + \sum_{i=1}^n s_{e_i}s_{e_i}^*$. Recall that $\nu$ is a maximal dilation therefore it extends to a $*$-representation of $\O_\G$. Hence, by applying $\nu$ we obtain
\begin{align*}
\begin{bmatrix} \rho(p_v) & 0 \\ 0 & \ast \end{bmatrix}
=
\begin{bmatrix} \rho(s_e) & a \\ 0 & \ast \end{bmatrix}
\cdot
\begin{bmatrix} \rho(s_e)^* & 0 \\ a^* & \ast \end{bmatrix}
+
\sum_{i=1}^n
\begin{bmatrix} \rho(s_{e_i}) & a_i \\ 0 & \ast \end{bmatrix}
\cdot
\begin{bmatrix} \rho(s_{e_i})^* & 0 \\ a_i^* & \ast \end{bmatrix}
\end{align*}
By equating the $(1,1)$-entries in the above equation we get
\[
\rho(p_v) = \rho(s_e)\rho(s_e)^* + aa^* + \sum_{i=1}^n \left(\rho(s_{e_i}) \rho(s_{e_i})^* + a_ia_i^*\right).
\]
By assumption $\rho$ is in turn a representation of $\O_\G$, hence
\[
\rho(p_v) = \rho(s_e)\rho(s_e)^* + \sum_{i=1}^n \rho(s_{e_i}) \rho(s_{e_i})^*.
\]
Thus $aa^* + \sum_{i=1}^n a_ia_i^*=0$, hence $a=0$. Therefore
\[
\nu(s_e) = \begin{bmatrix} \rho(s_e) & 0 \\ 0 & \ast \end{bmatrix}
\]
for every $e\in \G^{(1)}$.

Since $\T_\G^+$ is generated by $p_v$, for $v\in \G^{(0)}$, and $s_e$, for $e\in \G^{(1)}$, we obtain that $H$ is $\nu(\T_G^+)$-invariant and the proof is complete.
\end{proof}

A second example of tensor algebras is Peters' semicrossed product associated to a $\ca$-dynamical system $(A,\ga)$. Semicrossed products were initiated by Arveson \cite{Arv1} and formally defined by Peters \cite{Pet2}. They have been investigated by various authors \cite{ArvJ,McAsMuhSai79,HadH,Pow,MS00,DK08,DKsimple,Dun08,KakKat10,DavKak11} and Muhly and Solel \cite{MS} give the connection with a $\ca$-correspondence structure.

\begin{example}
\textup{Given a $*$-endomorphism $\ga$ of a $\ca$-algebra $A$ let the $\ca$-correspondence $X_A$, where $X=A$ is the trivial Hilbert $A$-module and the left action is defined by $\phi_X(a)(\xi)=\ga(a)\xi$ for all $a\in A$, $\xi \in X_A$.
The tensor algebra $\T_{X_A}^+ \equiv A \times_\ga \bbZ_+$ is \emph{Peter's semicrossed product of $A$ by $\ga$}.}

\textup{In view of Lemma \ref{L:equiv via comp} and Theorem \ref{T:main 4}, a semicrossed product is Dirichlet if and only if $\ker\ga$ is orthocomplemented in $A$ and $\ga$ is onto $A$, since $\K(A)=A$.  In particular, when $\ga$ is injective we deduce that the semicrossed product is Dirichlet if and only if $\ga$ is onto (thus a $*$-isomorphism). Therefore \cite[Proposition 3]{Dun08} which states that semicrossed products of injective (not necessarily onto) dynamical systems are Dirichlet, is false.}
\end{example}

Nevertheless the semicrossed products of injective dynamical systems have the unique extension property. Recall that an injective dynamical system $(A,\ga)$ extends to the automorphic dynamical system $(A_\infty,\ga_\infty)$ \cite{Sta93}
\begin{align*}
 \xymatrix{
  A \ar[r]^{\ga} \ar[d]^\ga &
  A \ar[r]^{\ga} \ar[d]^\ga &
  A \ar[r]^{\ga} \ar[d]^\ga &
  \cdots \ar[r] &
  A_\infty \ar[d]^{\ga_\infty} \\
  A \ar[r]^\ga &
  A \ar[r]^\ga &
  A \ar[r]^\ga &
  \cdots \ar[r] &
  A_\infty
 }
\end{align*}
Then by \cite[Theorem 2.5]{KakKat10} the $\ca$-envelope of $A\times_\ga \bbZ_+$ is the usual crossed product $A_\infty \rtimes_{\ga_\infty} \bbZ$. The proof of the next result is due to Elias Katsoulis.

\begin{theorem}\label{T:main 5}
Let $(A,\ga)$ be an injective unital $\ca$-dynamical system. Then the semicrossed product $A\times_{\ga} \bbZ_+$ has the unique extension property.
\end{theorem}
\begin{proof}
Since $(A,\ga)$ is unital, then $A\times_\ga \bbZ_+$ is generated by $A$ and the unitary $U$ in the crossed product $A_\infty \rtimes_{\ga_\infty} \bbZ$. Let $\rho\colon A_\infty \rtimes_{\ga_\infty} \bbZ \rightarrow \B(H)$ be a faithful $*$-representation. Then $\rho(U)$ is again a unitary. If $\nu$ is a dilation of $\rho|_{A \times_\ga \bbZ_+}$, then $\nu(U)$ is a dilation of the unitary $\rho(U)$ hence trivial. Also $\nu|_A$ is a dilation of the $*$-representation $\rho|_A$. Hence in both cases $H$ is $\nu(A\times_\ga \bbZ_+)$-invariant, thus $\rho$ is maximal \cite{DrMc05}.
\end{proof}

The same result is obtained for a third sub-class of independent interest. Let $\{\ga_i\}$ be a family of $n$ $*$-endomorphisms of a $\ca$-algebra $A$. The associated $\ca$-correspondence $X_{(A,\ga)}$ is the interior direct sum $\oplus_i A$ where the left action is defined by $\phi_{X_{(A,\ga)}}(a)(\oplus_i \xi_i)= \oplus_i (\ga_i(a)\xi_i)$, for $a\in A$ and $\xi_i \in A$ \cite{MS,DavKat}. Note that $X_{(A,\ga)}$ admits an orthogonal basis $\{e_i\}_{i=1}^n$, i.e., a vector $\xi \in X_{(A,\ga)}$ is written as an orthogonal sum $\sum_i \xi_i e_i$, for some $\xi_i \in A$.

We remark that recently Cornelissen and Marcolli use the theory of multivariable (automorphic) dynamical systems of Davidson and Katsoulis \cite{DavKat} in the context of number theory (see \cite[Theorem 2]{C} and \cite[Section 6]{CM}) and the reconstruction of graphs (see \cite[Theorem 1.5]{CM2} and the proof of \cite[Theorem 1.6]{CM2}).

\begin{theorem}\label{T:auto}
Let $\{\ga_i\}$ be a family of automorphisms of $A$. Then the tensor algebra of $X_{(A,\ga)}$ has the unique extension property.
\end{theorem}
\begin{proof}
First notice that any element $\xi \in X_{(A,\ga)}$ can be written as a linear combination of some $\xi_i e_i$ where every $\xi_i \in A^+$. Since $\ga_i$ are onto, then for any $\xi \xi^* \in A^+$, there are $a\in A$ and $\xi_i \in A$ for $i=2,\dots, n$, such that
\[
\phi_X(a) = \Theta^X_{\xi e_1,\xi e_1} + \sum_{i=2}^n \Theta^X_{\xi_i e_i, \xi_i e_i};
\]
namely $a = \ga_1^{-1}(\xi \xi^*)$ and $\xi_i\xi_i^*=\ga_i(\ga_1^{-1}(\xi\xi^*)$. The rest of the proof follows as in Theorem \ref{T:graphs}, repeated for $i=1, 2, \dots, n$.
\end{proof}

The combination of the proofs of Theorem \ref{T:graphs} and Theorem \ref{T:auto} imply the following result. The proof is left to the reader.

\begin{theorem}
Let $X_\G$ be a graph correspondence over $A_\G$ associated to a family $\big\{\G, \{A_p\}_{p \in \vrt}, \{X_e\}_{e\in \edg} \big\}$,
such that $\G$ is a row-finite graph and every $X_e$ is an equivalence bimodule. Then the tensor algebra of $X_\G$ has the unique extension property.
\end{theorem}

\subsection{The Counterexample}

In what follows we give an example of a Dirichlet algebra that is not a tensor algebra of a $\ca$-correspondence. This answers the question raised in the context of Davidson and Katsoulis \cite{DavKat11-2} concerning tensor algebras and semi-Dirichlet algebras.

\begin{proposition}\label{P:C-corre}
A $\ca$-correspondence $X$ over $\bbC$ is a Hilbert bimodule if and only if $X$ is the trivial $\ca$-correspondence $\bbC$ over $\bbC$.
\end{proposition}
\begin{proof}
If $X$ is a Hilbert bimodule over $\bbC$, then the left action $\phi_X\colon \bbC \rightarrow \L(X)$ is a $*$-isomorphism onto $\K(X)$, by Lemma \ref{L:equiv via comp}.  If $X$ contained two elements $\xi$ and $\eta$ such that $\xi \notin \bbC \eta$, then $\K(X)$ would contain the linearly independent operators $\Theta^X_{\xi,\eta}$ and $\Theta^X_{\eta,\eta}$, which is a contradiction. Thus $X$ is generated linearly by one element. Therefore $X \simeq \bbC$, hence $\L(X) \simeq \bbC$ and $\phi_X = \id_\bbC$. The converse is trivial.
\end{proof}

For a compact subset $K$ of $\bbC$, let $P(K)$ be the closed algebra of the polynomials supported on $K$ and $R(K)$ be the closed algebra of the rational functions that have poles in $\bbC \setminus K$. A subalgebra of $\rC(K)$ is called \emph{a uniform algebra on $K$} if it contains the constant functions and it separates the points.

\begin{example}
\textup{The most common example of a polynomial algebra that is Dirichlet is the disc algebra $A(\bbD)$ whose $\ca$-envelope is $\mathrm{C}(\bbT)$. It is trivial to check that $A(\bbD)$ is also the tensor algebra of the trivial Hilbert bimodule $\bbC$.}
\end{example}

For our counterexample, let $K$ be the compact subset of $\bbC$ defined by
\[
K:= \{z\in \bbC: |z-1| \leq  1 \} \cup \{z\in \bbC: |z+1|\leq 1 \},
\]
and fix $\A:=P(K)$. Abstractly, though we won't need it, the completely contractive representations of $\A$ are induced by operators on a Hilbert space whose spectrum is contained in $K$.

\begin{proposition}
Let $K$ be as above and $\A:=P(K)$. Then the $\ca$-envelope of $\A$ is $\rC(\partial K)$ and $\A$ is a Dirichlet algebra.
\end{proposition}
\begin{proof}
The set $\bbC \setminus K$ is connected, hence by Runge's Theorem $\A:=P(K)=R(K)$. Moreover $P(K)$ is a uniform subalgebra of $\rC(K)$. Therefore there is a unique compact subset $Y$ of $K$ such that $\cenv(\A)=\rC(Y)$, where the embedding $\iota\colon \A \rightarrow \mathrm{C}(Y)$ is given by $\iota(f)=f|_Y$ (see \cite[Corollary 15.17]{Pau02}). By the maximal modulus principle, the set $Y$ is contained in $\partial K$. The fact that $\A =R(K)$ implies that $Y$ cannot be a proper subset of $\partial K$. Thus $\cenv(\A)\simeq \rC(\partial K)$.

Finally, the set $\bbC \setminus K$ contains finitely many components and the interior of $K$ is the union of two simply connected sets. Therefore $\A=R(K)$ is Dirichlet \cite{Con}.
\end{proof}

We will show that there is not a $\ca$-correspondence $X_A$ such that its tensor algebra $\T_X^+$ is completely isometrically isomorphic to $\A : = P(K)$. To reach contradiction, suppose that there is a completely isometric isomorphism $\rho\colon \T_X^+\rightarrow \A$ for some $\ca$-correspondence $X_A$.

The restriction of $\rho$ to $A \subseteq \T_X^+$ is an injective $*$-homomorphism, as a completely isometric homomorphism of a $\ca$-algebra. Note that, by construction $\A \cap \A^* = \bbC$, hence the only $\ca$-algebra in $\A$ is $\bbC$. Therefore $A$ must coincide with $\bbC$ via the restriction of $\rho$. Hence $X$ is a Hilbert $\bbC$-module.

Moreover $\T_X^+$ is a Dirichlet algebra being complete isometrically isomorphic to $\A$. Thus, by Theorem \ref{T:main 3}, the $\ca$-correspondence $X_A$ is a Hilbert bimodule, hence $X=\bbC$ by Proposition \ref{P:C-corre}. Therefore $\T_X^+ \simeq A(\bbD)$.

In particular the $\ca$-envelopes $\rC(\partial K) \simeq \cenv(\A)$ and $\rC(\bbT) \simeq \cenv(A(\bbD))$ must be $*$-isomorphic. This induces a homeomorphism between $\partial K$ and $\bbT$. Thus $\partial K \setminus \{0\}$ is homeomorphic to the interval $(0,1)$. The contradiction then follows since $\partial K$ is the union of the circles $C\big( (1,0), 1 \big)$ and $C\big( (-1,0), 1 \big)$, hence $\partial K \setminus \{0\}$ is not connected.

\begin{remark}
\textup{In the same way, other counterexamples can be constructed by considering $K$ to be a union of finitely many compact subsets of $\bbC$ with no holes, such that every pair intersects at a single point on their boundaries and $\bbC \setminus K$ is connected.}
\end{remark}

\begin{acknow}
The author would like to emphasize that this paper initially started as a joint project with Elias Katsoulis. Nevertheless, following the suggestion of Elias Katsoulis, it was decided for the paper to go as single-authored. The author would like to thank Elias Katsoulis for his suggestion and permission to include here material resulting from conversations in the initial stages of that project.
\end{acknow}


\end{document}